\documentclass[a4paper,11pt]{article}

\usepackage[top=3.0cm,bottom=3.0cm,left=2.5cm,right=2.5cm]{geometry}%{geometry} ???????????
\usepackage{tikz}
\usepackage{graphics,epsfig,psfrag}
\usepackage{graphicx}
\usepackage{amsfonts}
\usepackage{amssymb}
\usepackage{amsmath}
\usepackage{amsthm}
\usepackage[english]{babel}%??%??latex默认字体 Latin Modern Math
\usepackage{ctex}%使得可以识别中文
\usepackage{hyperref}%让文中引用的地方用红色或者蓝色方框标明
\usepackage{cases}%\begin{numcases}{|x|=}x,&for$x\geq0$\\-x,&for$x<0$\end{numcases}
\usepackage{authblk}%作者的单位出现在名字下方，且不影响脚注
\usepackage{mathtools}%让文中公式命令flalign,align,equation可以实现
\usepackage{pifont}%让公式可以用带圈的数字进行标号

\newtheorem{lemma}{Lemma}[section]
\newtheorem{theorem}[lemma]{Theorem}
\newtheorem{corollary}[lemma]{Corollary}

\newtheorem{claim}[]{\noindent Claim}[section]

\begin{document}

\title{Characterization of $P_3$-connected graphs}

        \author[1]{Rong Chen\footnote{Email: rongchen@fzu.edu.cn.}}
		%\author[1]{Zijian Deng\footnote{Email: zj1329205716@163.com.(corresponding author).}}

\affil[1]{Center for Discrete Mathematics, Fuzhou University\\

Fuzhou, People's Republic of China}

	\date{}
	\maketitle
	\begin{abstract}
For any pair of edges $e,f$ of a graph $G$, we say that {\em $e,f$ are $P_3$-connected in $G$} if there exists a sequence of edges $e=e_0,e_1,\ldots, e_k=f$ such that $e_i$ and $e_{i+1}$ are two edges of an induced $3$-vertex path in $G$ for every $0\leq i\leq k-1$. If every pair of edges of $G$ are $P_3$-connected in $G$, then $G$ is {\em $P_3$-connected}. $P_3$-connectivity was first defined by Chudnovsky et al. in 2024 to prove that every connected graph not containing $P_5$ as an induced subgraph has cop number at most two. In this paper, we give a characterization of $P_3$-connected graphs and prove that a simple graph is $P_3$-connected if and only if it is connected and has no homogeneous set whose induced subgraph contains an edge.
	\end{abstract}
	
%	\textbf{Mathematics Subject Classification}:
	

%\footnote{This research was partially supported by grants from the National Natural Sciences Foundation of China (No. 11971111).}

%%%%%%%%%%%%%%%%%%%%%%%%%%%%%%%%%%%%%%%%%%%%%%%%%%%%%%%%%%%%%%%%%%%%%%%%%%%%%%%%%%%%%%%%%%%%%%%%%%%%%%%%%%%
%%%%%%%%%%%%%%%%%%%%%%%%%%%%%%%%%%%%%%%%%%%%%%%%%%%%%%%%%%%%%%%%%%%%%%%%%%%%%%%%%%%%%%%%%%%%%%%%%%%%%%%%%%%
\section{\bf Introduction}
All graphs considered in this paper are finite and simple. Let $P_n$ be a path with exactly $n$ vertices. Let $X$ be a subset of a graph $G$. When $G[X]$ contains no edge, we say that $X$ is \emph{stable}. If $2\leq|X|<|V(G)|$ and for any vertex $y\in V(G)-X$, either $X\subseteq N(y)$ or $X\cap N(y)=\emptyset$, then  we say that $X$ is a {\em homogenous set} of $G$. For any pair of edges $e,f$ of $G$, we say that {\em $e,f$ are $P_3$-connected in $G$} if there exists a sequence of edges $e=e_0,e_1,\ldots, e_k=f$ such that $e_i$ and $e_{i+1}$ are two edges of an induced $P_3$ in $G$ for every $0\leq i\leq k-1$. If every pair of edges of $G$ are $P_3$-connected in $G$, then $G$ is {\em $P_3$-connected}. $P_3$-connectivity was first defined by Chudnovsky et al. in \cite{CNPT} to prove that every connected graph not containing $P_5$ as an induced subgraph has cop number at most two. In this paper, we give a characterization of $P_3$-connected graphs and prove the following result.

%For a graph $G$, we use $\chi(G)$, $\omega(G)$ and $\Delta(G)$ to denote the chromatic number, clique number and maximum degree of $G$, respectively. A \emph{stable set} %(or a \emph{stable set}) is a subset of $V(G)$ that are pairwise nonadjacent.

\begin{theorem} \label{main result}
A graph is $P_3$-connected if and only if it is connected and has no non-stable homogeneous set.
\end{theorem}

By Theorem \ref{main result} and the definition of homogeneous sets, we have

\begin{corollary} %\label{main result}
Each connected triangle-free graph is $P_3$-connected.
\end{corollary}

%%%%%%%%%%%%%%%%%%%%%%%%%%%%%%%%%%%%%%%%%%%%%%%%%%%%%%%%%%%%%%%%%%%%%%%%%%%%%%%%%%%%%%%%%%%%%%%%%%%%%%%%%%%
%%%%%%%%%%%%%%%%%%%%%%%%%%%%%%%%%%%%%%%%%%%%%%%%%%%%%%%%%%%%%%%%%%%%%%%%%%%%%%%%%%%%%%%%%%%%%%%%%%%%%%%%%%%
\section{Proof of Theorem \ref{main result}} %Finding a large complete bipartite minor }
%%%%%%%%%%%%%%%%%%%%%%%%%%%%%%%%%%%%%%%%%%%%%%%%%%%%%%%%%%%%%%%%%%%%%%%%%%%%%%%%%%%%%%%%%%%%%%%%%%%%%%%%%%%
%%%%%%%%%%%%%%%%%%%%%%%%%%%%%%%%%%%%%%%%%%%%%%%%%%%%%%%%%%%%%%%%%%%%%%%%%%%%%%%%%%%%%%%%%%%%%%%%%%%%%%%%%%%
%For a graph $G$ and a subset $X$ of $V(G)$, let $G-X$ denote the graph obtained from $G$ by deleting all vertices in $X$ and let $G[X]$ be the subgraph of $G$ induced by $X$.

For a graph $G=(V,E)$, if $X\subseteq V$, then we use $G[X]$ and $G\setminus X$ to denote the subgraphs of $G$ induced by $X$ and $V\setminus X$, respectively. When $X= \{ x\}$, we write $G\setminus x$ instead of $G\setminus\{x \}$.
Let $N(X)$ be the set of vertices in $V(G)-X$ that have a neighbour in $X$. Set $N[X]:=N(X)\cup X$.
We say that $X$ is {\em connected} when $G[X]$ is connected. If $X$ is connected and $X=V(G)$, we say that $X$ is {\em spanning}. %We say that $X$ is {\em spanning} if $X=V(G)$ and $G[X]$ is connected.
%For any $x\in V$, let $N_G(x)$ denote the set of all neighbors of $x$ in $G$. Set $N_G[x]:=N_G(x)\cup\{x\}$.
%When there is no confusion, the subscripts will be omitted. % and we often write $N(x)$ and $d(x)$ if the context is clear.
%We use $|G|$ to denote the number of vertices, $\omega(G)$ the clique number, and $\Delta(G)$ the maximum degree.
For $A,B\subseteq V$, we write $[A,B]$ to denote the subgraph of $G$ induced by the set of edges that have one end in $A$ and another end in $B$. Note that a vertex in $A$ has no neighbour in $B$ is not a vertex in the graph $[A,B]$. We say that $A$ is {\em complete} to $B$ if every vertex in $A$ is adjacent to every vertex in $B$, and that $A$ is {\em anti-complete} to $B$ if $E([A,B])=\emptyset$. If $2\leq|A|<|V|$ and for any vertex $x\in V-A$, the vertex $x$ is either complete to $A$ or anti-complete to $A$, then we say that $A$ is a {\em homogeneous set} of $G$. Evidently, if $G$ has no non-stable homogeneous set, then it has no connected homogeneous set. This fact will be frequently used in the proof of Theorem \ref{main result}.

For any pair of edges $e,f$ of $G$, we say that $e$ is {\em $P_3$-connected to $f$} or {\em $e,f$ are $P_3$-connected in $G$} if there exists a sequence of edges $e=e_0,e_1,\ldots, e_k=f$ such that $e_i$ and $e_{i+1}$ are two edges of an induced $P_3$ in $G$ for every $0\leq i\leq k-1$. For any subgraph $H$ of $G$, if every pair of edges $e,f$ of $H$ are $P_3$-connected in $G$, we say that $H$ is {\em $P_3$-connected in $G$}. When $G$ is $P_3$-connected in $G$, we say that $G$ is {\em $P_3$-connected.} Note that, when $H$ is $P_3$-connected in $G$, the graph $H$ maybe not connected by definition. However, we have

\begin{lemma} \label{P3-connected}
If a graph $G$ is $P_3$-connected, then $G$ is connected and has no non-stable homogeneous set.
\end{lemma}
\begin{proof}
%By Lemma \ref{main lem}, to prove the theorem, it suffices to show that the ``only if" part holds. Let $G$ be a simple $P_3$-connected graph.
Since two edges in different components of $G$ can not be $P_3$-connected in $G$ by the definition of $P_3$-connectivity, $G$ is connected. Assume to the contrary that $X$ is a non-stable homogeneous set of $G$. Let $e$ be  any edge of $G[X]$. Since $X$ is non-stable, such $e$ exists. Since no induced $P_3$-path $P$ contains $e$ such that $V(P)-V(\{e\})\notin X$, no edge in $G[X]$ can be $P_3$-connected to an edge with at least one end in $V(G)-X$, so $G$ is not $P_3$-connected, a contradiction. Hence, $G$ has no homogeneous non-stable set.
\end{proof}

Next, we will prove that if a graph is connected and has no non-stable homogeneous set, then it is $P_3$-connected.

When $e$ is $P_3$-connected to $f$ in $G$, and $f$ is $P_3$-connected to $g$ in $G$, it is obvious that $e$ is $P_3$-connected to $g$ in $G$. Hence, there is a partition $(E_1,E_2,\ldots, E_m)$ of $E(G)$ such that all $G|E_i$ are edge-maximal $P_3$-connected subgraphs in $G$. Evidently, the graph $G$ is $P_3$-connected if and only if $m=1$. By the definition of $P_3$-connectivity, each edge-maximal $P_3$-connected subgraph of $G$ is connected.

\begin{lemma} \label{spanning}
If a connected graph $G$ does not contain a connected homogeneous set, then each edge-maximal $P_3$-connected subgraph of $G$ is spanning.
\end{lemma}
\begin{proof}
Assume not. Let $H$ be a edge-maximal $P_3$-connected subgraph of $G$ that is not spanning. Then $|V(H)|\geq2$ as $H$ contains at least one edge. Since $H$ is connected, we have $V(H)\subsetneqq V(G)$. For any vertex $u\in N(H)$, if $u$ is not complete to $H$, then $G$ contains an induced $u$-$v$-$w$ path with $vw\in E(H)$ as $H$ is connected, so $G|(E(H)\cup\{uv\})$ is $P_3$-connected in $G$, which contradicts to the fact that $H$ is edge-maximal. Hence, $N(H)$ is complete to $H$ by the arbitrary choice of $u$. Then $V(H)$ is a connected homogeneous set as $|V(H)|\geq2$, a contradiction.
\end{proof}

We say that a graph $G$ is {\em anti-connected} or an {\em anti-path} if the complement of $G$ is connected or an path. Similarly, we say a subset $X$ of $V(G)$ is an {\em anti-component} of $G$ if $X$ is the vertex set of a component of the complement of $G$.

\begin{lemma}\label{complete adjacent}
Let $X,Y$ be disjoint vertex subsets of a graph $G$. If $X$ is complete to $Y$ and $G[X], G[Y]$ are anti-connected, then $[X,Y]$ is $P_3$-connected in $G$.
\end{lemma}
\begin{proof}
Let $y\in Y$ and $u,v$ be neighbours of $y$ in $X$. Since $X$ is anti-connected, there is an anti-path $u=x_0$-$x_1$-$\cdots$-$x_m=v$ in $G[X]$. Since $x_ix_{i+1}\notin E(G)$ and $y$ is complete to $X$, $x_i$-$y$-$x_{i+1}$ is an induced $P_3$ for all $0\leq i\leq m-1$, so $yu$ is $P_3$-connected to $yv$ in $G$. Hence,  by the arbitrary choice of $u,v$, the subgraph $[X,y]$ is $P_3$-connected in $G$ for each $y\in Y$. By symmetry, each $[x,Y]$ is also $P_3$-connected in $G$ for eah $x\in X$. Hence, $[X,Y]$ is $P_3$-connected in $G$ as $X$ is complete to $Y$.
\end{proof}

\begin{lemma} \label{main lem}
If a connected graph $G$ has no non-stable homogeneous set, then $G$ is $P_3$-connected.
\end{lemma}
\begin{proof}
Assume not. Then there is a partition $(E_1,E_2,\ldots, E_m)$ of $E(G)$ with $m\geq2$ such that all $G|E_i$ are edge-maximal $P_3$-connected subgraphs in $G$. So $|V(G)|\geq3$ as $G$ is simple.%By Lemma \ref{spanning},
\begin{claim}\label{Ei}
For each $1\leq i\leq m$, the subgraph $G|E_i$ is a spanning subgraph of $G$.
\end{claim}
\begin{proof}
Since the subgraph induced by each connected homogeneous sets contain at least one edge  and $G$ has no non-stable homogeneous set, $G$ has no connected homogeneous set. Hence, Claim \ref{Ei} follows immediately from Lemma \ref{spanning}.
\end{proof}

Since $G$ has no non-stable homogeneous set and $|V(G)|\geq3$, the graph $G$ is not a clique.
Arbitrary choose a vertex $x\in V(G)$ with $V(G)-N[x]\neq\emptyset$. Set
\[Y:=V(G)-N[x],\ \ X_i:=\{u\in N(x):\ xu\in E_i\},\] for any integer $1\leq i\leq m$.
%Since $G$ has no non-stable homogeneous set, $G$ has no universal vertex, so $Y\neq\emptyset$.
By the definition of $P_3$-connectivity, since $G$ is simple, $(X_1,X_2,\ldots, X_m)$ is a partition of $N(x)$, and $X_i$ is complete to $X_j$ for all $1\leq i<j\leq m$. Let $(X_{i1},X_{i2},\ldots, X_{im_i})$ be the partition of $X_i$ such that all $X_{ij}$ are anti-components of $G[X_i]$. That is, $X_{ij}$ is complete to $X_{ik}$ for all $1\leq j<k\leq m_i$. Hence,
\begin{claim}\label{complete}
For any integers $1\leq i\leq j\leq m$, $1\leq s\leq m_i$ and $1\leq t\leq m_j$, the set $X_{is}$ is complete to $X_{jt}$.
\end{claim}

%Moreover, since $G$ is anti-connected,
%\begin{claim}\label{}
%No $X_{ij}$ is complete to $Y$.
%\end{claim}
\begin{claim}\label{color}
For any integer $1\leq i\leq m$, we have $[Y,X_i]\subseteq E_i$.
\end{claim}
\begin{proof}
For any edge $yx'\in E([Y,X_i])$ with $y\in Y$ and $x'\in X_i$, since $y$-$x'$-$x$ is an induced 3-vertex path and $xx'\in E_i$, we have $yx'\in E_i$, so the claim holds by the arbitrary choice of $yx'$.
\end{proof}

For any integers $1\leq i\leq m$ and $1\leq j\leq m_i$, set \[Y_{ij}:=\{y\in Y:\ y\ \text{has a neighbour in}\ X_{ij}\}.\] Then we have
\begin{claim}\label{Yij}
For any integers $1\leq i\leq m$ and $1\leq j\leq m_i$, the set $X_{ij}$ is anti-complete to $Y-Y_{ij}$.
\end{claim}
\begin{claim}\label{color+}
For any integers $1\leq i\leq m$ and $1\leq j<k\leq m_i$, we have $[X_{ij},X_{ik}]\subseteq E_i$.
\end{claim}
\begin{proof}
By Lemma \ref{complete adjacent} and Claim \ref{complete}, to prove the claim it suffices to show that some edge in $[X_{ij},X_{ik}]$ is in $E_i$. Assume that some $u\in X_{ik}$ has a non-neighbour $y$ in $Y_{ij}$. By the definition of $Y_{ij}$, there is a vertex $v\in N(y)\cap X_{ij}$. Then $yv\in E_i$ by Claim \ref{color}. Moreover, since $y$-$v$-$u$ is an induced path by Claim \ref{complete}, $uv\in E_i$. %Hence, Claim \ref{color+} holds from Lemma \ref{complete adjacent} and Claim \ref{complete}.
So we may assume that $X_{ik}$ is complete to $Y_{ij}$, implying $Y_{ij}\subseteq Y_{ik}$. By symmetry, $Y_{ik}\subseteq Y_{ij}$ and $X_{ij}$ is complete to $Y_{ik}$. Hence, $Y_{ij}=Y_{ik}$, and $X_{ij}\cup X_{ik}$ is complete to $Y_{ij}$. Since $X_{ij}\cup X_{ik}$ is anti-complete to $Y-Y_{ij}$ by Claim \ref{Yij}, the set $X_{ij}\cup X_{ik}$ is a homogeneous set of $G$ that is connected by Claim \ref{complete}, which is a contradiction to the fact that $G$ has no non-stable homogeneous set.
\end{proof}

\begin{claim}\label{color+1}
For any integer $1\leq i< j\leq m$, $1\leq s\leq m_i$ and $1\leq t\leq m_j$, exactly one of the following holds.
\begin{itemize}
\item[(1)] $[X_{is},X_{jt}]\subseteq E_j$, and $X_{jt}$ is complete to $Y_{is}$, implying $Y_{is}\subseteq Y_{jt}$.
\item[(2)] $[X_{is},X_{jt}]\subseteq E_i$, and $X_{is}$ is complete to $Y_{jt}$, implying $Y_{jt}\subseteq Y_{is}$.
\end{itemize}
\end{claim}
\begin{proof}
Since $E_i\cap E_j=\emptyset$, (1) and (2) can not happen at same time. Hence, to prove the claim is true, it suffices to show that (1) or (2) holds.
Note that, following a similar way as the proof of Claim \ref{color+}, when $X_{jt}$ is not complete to $Y_{is}$, some edge in $[X_{is},X_{jt}]$ is in $E_i$ by Claim \ref{color}, implying $[X_{is},X_{jt}]\subseteq E_i$ by Lemma \ref{complete adjacent}; and when $X_{is}$ is not complete to $Y_{jt}$, some edge in $[X_{is},X_{jt}]$ is in $E_j$, implying $[X_{is},X_{jt}]\subseteq E_j$. Hence, either $X_{jt}$ is complete to $Y_{is}$ or $X_{is}$ is complete to $Y_{jt}$.
Without loss of generality we may assume that $X_{jt}$ is complete to $Y_{is}$, implying $Y_{is}\subseteq Y_{jt}$.
When some vertex in $X_{is}$ has a non-neighbour in $Y_{is}$, following a similar way as the proof of Claim \ref{color+} again, we have $[X_{is},X_{jt}]\subseteq E_j$, so (1) holds. Hence, we may assume that $X_{is}$ is complete to $Y_{is}$. Since $X_{is}\cup X_{jt}$ is not a connected homogeneous set of $G$, some vertex $u\in X_{jt}$ has a neighbour $y$ in $Y-Y_{is}$ by Claim \ref{Yij}, so $y$-$u$-$v$ is an induced path for any $v\in X_{is}$ by Claims \ref{complete} and \ref{Yij}. Since $yu\in E_j$ by Claim \ref{color}, we have $[X_{is},X_{jt}]\subseteq E_j$ by Lemma \ref{complete adjacent}. That is, (1) holds. This proves Claim \ref{color+1}.
%Assume that $X_{jt}$ is anti-complete to $Y-Y_{is}$. Since $G$ has no homogeneous set, $|X_{jt}|=1$, implying $N[x]\subseteq N[X_{jt}]$, so the edge linking $x$ and the vertex in $X_{jt}$ is a domineering edge, a contradiction.
\end{proof}

Let $D$ be a directed graph with vertex set $\{X_{is}:\ 1\leq i\leq m,\ 1\leq s\leq m_i\}$.
For any integers $1\leq i<j\leq m$, $1\leq s\leq m_i$ and $1\leq t\leq m_j$, the vertex $X_{is}$ is directed to $X_{jt}$ if Claim \ref{color+1} (1) happens, %$X_{jt}$ is complete to $Y_{is}$,
and $X_{jt}$ is directed to $X_{is}$ if Claim \ref{color+1} (2) happens. %$X_{is}$ is complete to $Y_{jt}$. We claim that $D$ is acyclic.
Assume that $D$ has a directed cycle $C$. By Claim \ref{color+1}, the neighbourhood $Y_{is}$ of all vertices $X_{is}$ in $V(C)$ are the same, and $\bigcup_{X_{is}\in V(C)}X_{is}$ is complete to $Y_{is}$, so $\bigcup_{X_{is}\in V(C)}X_{is}$ is a connected homogenous set of $G$, which is a contradiction. So $D$ is acyclic.

Since $D$ is acyclic, there is a vertex $X_{is}$ of $D$ whose out-degree is zero. %such $X_{is}$ exists.
Then $[X_{is},X_{jt}]\subseteq E_i$ for any integers $1\leq j\neq i\leq m$ and $1\leq t\leq m_j$ by the definition of $D$ and Claim \ref{color+1}. Moreover, by Claims \ref{color} and \ref{color+}, $[X_{is},V(G)-X_{is}]\subseteq E_i$. Hence, $G|E_j$ is not spanning for any $j$ with $1\leq j\neq i\leq m$, a contradiction to Claim \ref{Ei}.
%By Claim \ref{color+1}, there is some $X_{is}$ such that $[X_{is},X_{jt}]\subseteq E_i$, and $X_{is}$ is complete to $Y_{jt}$ for any integer $1\leq j\neq i\leq m$ and $1\leq t\leq m_j$.
\end{proof}

\begin{proof}[Proof of Theorem \ref{main result}.]
Theorem \ref{main result} follows immediately from Lemmas \ref{P3-connected} and \ref{main lem}.
\end{proof}

%\section{Acknowledgments}
%This research was partially supported by grants from the National Natural Sciences Foundation of China (No. 11971111).
%We would like to express our gratitude to two anonymous reviewers for their diligent review and valuable suggestions, which significantly enhanced the clarity and presentation of this paper.

%\section*{Declaration}
	
	%\noindent$\Large{\textbf{Declarations}}$

%\section{Declaration of competing interest}
%The author declare that they have no known competing financial interests or personal relationships that could have appeared to influence the work reported in this paper.

%\section{Data availability}
%No data was used for the research described in the article.

\end{document}